\documentclass[11pt, reqno]{amsart}
\usepackage{amsmath,amssymb,amsbsy,amsfonts,amsthm,latexsym,amsopn,amstext,amsxtra,epic, euscript,amscd,indentfirst,verbatim,graphicx, hyperref,xcolor, setspace,tikz,tikz-cd}
\usepackage{tabularx} 
\usepackage{relsize} 
\usepackage[margin=1.2in]{geometry}
\usepackage{enumitem} 
\usepackage{hyperref} 
\usepackage{fancyhdr}

\DeclareMathOperator{\dist}{dist}

\DeclareMathOperator{\Vol}{Vol}

\theoremstyle{plain}
\newtheorem{theorem}{Theorem}[section]
\newtheorem{defi}[theorem]{Definition}
\newtheorem{corollary}[theorem]{Corollary}
\newtheorem{lemma}[theorem]{Lemma}
\newtheorem{notation}[theorem]{Notation}
\newtheorem{remark}{Remark}
\newtheorem{question}[theorem]{Question}
\newtheorem{open}[theorem]{Open problem}
\newtheorem{claim}[theorem]{Claim}
\newtheorem{proposition}[theorem]{Proposition}

\newtheorem{fact}[theorem]{Fact}

\newcommand{\RR} {\mathbb R}

\newcommand{\NN} {\mathbb N}

\newcommand{\NNN}{\mathcal{N}}

\newcommand{\LLL}{\mathcal{L}}

\newcommand{\pa} {\partial}

\newcommand{\beq} {\begin{equation}}
\newcommand{\eeq} {\end{equation}}

\numberwithin{equation}{section}

\newcommand{\heart}{\ensuremath\heartsuit}

\newcommand{\SM}[1]{\textcolor{blue}{(SM: #1)}}

\begin{document}
\allowdisplaybreaks

\title{
Eigenfunction localization and nodal geometry on dumbbell domains}

\author[Indian Institute of Technology Bombay]{Saikat Maji} 
 \address[Saikat Maji]{Department of Mathematics, Indian Institute of Technology Bombay, Powai, Maharashtra 400076, India.}
  \email{saikatmaji1997@gmail.com}

  \author[Iowa State University]{Soumyajit Saha} 
\address[Soumyajit Saha]{Department of Mathematics, Iowa State University, Ames, Iowa 50011, USA.}  \email{ssaha1@iastate.edu}

\maketitle

\begin{abstract}
In this article, we study the location of the first nodal line and hot spots under different boundary conditions on dumbbell-shaped domains. Apart from its intrinsic interest, dumbbell domains are also geometrically contrasting to the extensively studied convex domains. 
For dumbbells with Dirichlet boundary, we investigate the location of the supremum level set 
of the first eigenfunction and discuss the optimal positioning of obstacles. 
Considering the other end of level sets, the nodal sets, we establish that the first nodal set of a Neumann dumbbell (with sufficiently narrow connectors) lies within a neighborhood of the connectors.
The article demonstrates the utilization of the asymptotic $L^2$-localization (or its absence, characterized by either Dirichlet or Neumann boundaries) of dumbbell domains in tackling the aforementioned nodal geometry problems.

\end{abstract}


\section{Introduction and overview}
Consider a bounded Euclidean domain $\Omega$ and the Laplace operator $\Delta = \sum_j \pa_j^2$ with either the Dirichlet boundary condition 
\begin{equation*}\label{eqtn: Dirichlet condition}
    \varphi(x)=0, \hspace{5pt} x\in \pa \Omega,
\end{equation*}
or the Neumann boundary condition 
\begin{equation*}\label{eqtn:Neumann_condition}
    \pa_\eta\varphi(x)=0, \hspace{5pt} x\in \pa \Omega,
\end{equation*}
where $\eta$ denotes the outward unit normal on $\pa \Omega$.
Recall that if $\Omega$ has a reasonably regular boundary, the Laplacian $-\Delta$ with Dirichlet or Neumann boundary has a discrete spectrum (repeated with multiplicity)
$$ 
0\leq\lambda_1 \leq \dots \leq \lambda_k \leq \dots \nearrow \infty,
$$
 with corresponding (real-valued $L^2$-normalized) eigenfunctions $\varphi_k$.  Also, recall that the first eigenvalue is always simple, and the first Dirichlet eigenvalue is non-zero while its Neumann counterpart is always zero.  Let  $\NNN(\varphi) := \{ x \in M: \varphi (x) = 0\}$ denote the nodal set corresponding to any eigenfunction $\varphi$. Any connected component of $M \setminus \mathcal{N}({\varphi})$ is referred to as a nodal domain of the eigenfunction $\varphi$. 

One of the extensively studied problems in nodal geometry revolves around determining the location of extrema (commonly referred to as hot or cold spots) of the first non-trivial eigenfunction within a domain with either Dirichlet or Neumann boundaries.
The celebrated {\em hot spot conjecture} by Rauch states that the extrema of the first non-constant Neumann eigenfunction is attained on the boundary (see \cite{BB} for different variations of the conjecture). On the contrary, for Dirichlet eigenfunctions, owing to the boundary condition, the extrema is always located in the interior of the domain. This leads to the question: \emph{can we gain more precise information on the location of the Dirichlet hot spots in relation to the domain?} In regards to the Neumann hot spots, significant progress has been made for domains with symmetry and certain convex domains (see \cite{JM, Ste} and references therein). 
In Section \ref{sec: obstactles} of this article, we address the Dirichlet case. In the following theorem, 
we delineate the location of the hot spot within a Dirichlet dumbbell in terms of the hot spot of its base domains.
\begin{theorem}\label{thm: main result}
Consider two bounded domains $\Omega_1, \Omega_2\subset\RR^n$, where $\Omega_1$ is convex and $\Omega_2$ is simply connected, and a one-parameter family of dumbbells $\Omega_\epsilon$ (as described in Section \ref{sec: construction}). 
Let $\lambda_1(\Omega_i)$ be the first eigenvalue of 
$\Omega_i$ ($i=1, 2$) and $\varphi_{1, \epsilon}, \varphi_1^{\Omega_1}$ denote the first eigenfunction of $\Omega_\epsilon, \Omega_1$ respectively. Moreover, assume that $\lambda_1(\Omega_1)< \lambda_1(\Omega_2)$ and $\varphi_1^{\Omega_1}$ is non-negative.  Given $\delta>0$ , there exists  $\epsilon'>0$ (depending on $\delta$) such that
$$\LLL(\Omega_\epsilon) \subset B(x_0, \delta)~~~ \text{for any}~~~ \epsilon\leq \epsilon',$$ 
where $\LLL(\Omega_\epsilon)$ denotes the set of maximum points of $\varphi_{1, \epsilon}$ and $x_0$ denotes the unique maximum point of $\varphi_1^{\Omega_1}$.
\end{theorem}
The above theorem tells us that the hot spots of the ground state Dirichlet eigenfunction of a dumbbell are located deep within one of its subdomains when the connector width is narrow enough. 
\begin{remark}
    The assumption that $\Omega_1$ is convex is not crucial. The above theorem is true for any domain $\Omega_1$ with a unique maximum point of the ground state Dirichlet eigenfunction.
\end{remark}

Understanding the location of Dirichlet hot spots also helps us address the following optimization problem: \emph{given a simply-connected domain, can we determine the optimal position for creating a hole so that the first eigenvalue of the perforated domain is maximized (or minimized)?} The hole is referred to as the \emph{obstacle.} In general, any bounded domain can be regarded as an obstacle. However, in this article, we will restrict ourselves to convex obstacles. More precisely, given an Euclidean domain $\Omega\subset\RR^n$ and a convex domain $D\subset \RR^n$, consider the family of domains
\begin{equation*}
    \mathcal{F}(\Omega, D)=\{\Omega\setminus \overline{(y+D)}: y+D \text{ is a translate of } D \text{ such that } (y+D)\subset \Omega \}.
\end{equation*}
The problem is to determine the optimal translate $y_0+D$ for which 
\begin{equation*}
   \max_{\mathcal{F}(\Omega, D)} \lambda_1(\Omega\setminus \overline{(y+D)}) =  \lambda_1(\Omega\setminus \overline{(y_0+D)}),
\end{equation*}
The problem was first investigated by Hersch in \cite{Hersch} for the planer annular domains (i.e., disc with a circular obstacle) and Kesavan in \cite{Kesavan} proved an analogous result in higher dimensions using the \emph{moving plane method}. In \cite{HarellKrugerKurata}, Harrell, Kröger, and Kurata studied the obstacle problem on domains with \emph{interior reflection property} which was later on generalized by Georgiev and Mukherjee in \cite{GM} on domains that satisfy a certain asymmetry condition. We also refer our readers to \cite{CR,AAK} (and references therein) for additional insight on the problem.  

As an application of the above Theorem \ref{thm: main result}, we have the following corollary which tells us that in order to find the optimal position of an obstacle in a dumbbell, it is enough to locate the max-point of the ground state eigenfunction for one of the base domains of the dumbbell.

\begin{corollary}\label{cor: obstacle} 
    Let $y_\epsilon+D$ be the optimal translation of the obstacle $D$ in the dumbbell $\Omega_\epsilon$ (as defined in Theorem \ref{thm: main result}) and $x_0$ denotes the unique maximum point of $\varphi_1^{\Omega_1}$. 
     Given $\delta>0$, there exists $\epsilon'>0$ such that for every $\epsilon<\epsilon'$, we have $d(x_0, y_\epsilon+D)<\delta$ whenever the obstacle is sufficiently large (in the sense of Theorem \ref{thm: GM result} below). Here $d(x, A)$ denotes the distance between $x$ and the set $A$. 
\end{corollary}

So far, we have only studied the fundamental frequency and ground state eigenfunction of a dumbbell, where we make use of the fact that the Dirichlet eigenfunctions of a dumbbell are asymptotically $L^2$-localized on one of the two base domains of a dumbbell. On the contrary, Neumann eigenfunctions can either $L^2$-localize on the base domains or the connector of the dumbbell, or not localize at all. In order to exhibit how such non-localizations affect certain nodal geometry problems on dumbbells,  we now turn our attention to second eigenfunctions, in particular, the nodal set corresponding to the second eigenfunctions. 

We start by recalling the nodal line conjecture by Payne \cite{Payne} which says that: \emph{any Dirichlet second eigenfunction of a bounded Euclidean planar domain cannot have a closed nodal line}. 
This has been extensively studied over the past few decades and we refer our readers to \cite{MS, MS1} for a detailed literature survey and recent developments. 
 However, the conjecture still remains open for simply-connected planar domains. Interestingly, considering simply-connected domains with the Neumann boundary condition, the proof for the conjecture is relatively straightforward (see Proposition \ref{prop: neumannn conj}) and was proved by Pleijel \cite{Pl} years before Payne conjectured it for the Dirichlet case. Then, the question is: \emph{can we find an approximate location of the nodal line corresponding to the second Neumann eigenfunctions?} In this regard, Jerison \cite{Je} proved that the location of the first nodal line for a convex planar domain is near the unique zero of the first non-constant eigenfunction of a certain ordinary differential operator. Atar and Burdzy in \cite{AtarBurdzy2002} took a probabilistic approach and used  ``couplings'' of Brownian motion to determine subregions of several (non-convex) simply-connected planar domains where the first nodal line must intersect.  

 In Section \ref{sec: Neumann dumbbells} of this article, we showcase how the lack of localization of certain Neumann eigenfunctions on dumbbells has an effect on the determination of the first nodal line's location. In particular, we prove the following 
\begin{theorem}[Informal]\label{thm: informal main result 2}
    The nodal line corresponding to the first non-constant Neumann eigenfunction of a dumbbell domain with a sufficiently narrow connector cannot enter too deep inside either of the two base domains of the dumbbell.
\end{theorem}
\begin{figure}[ht]
\centering
\includegraphics[height=4cm]{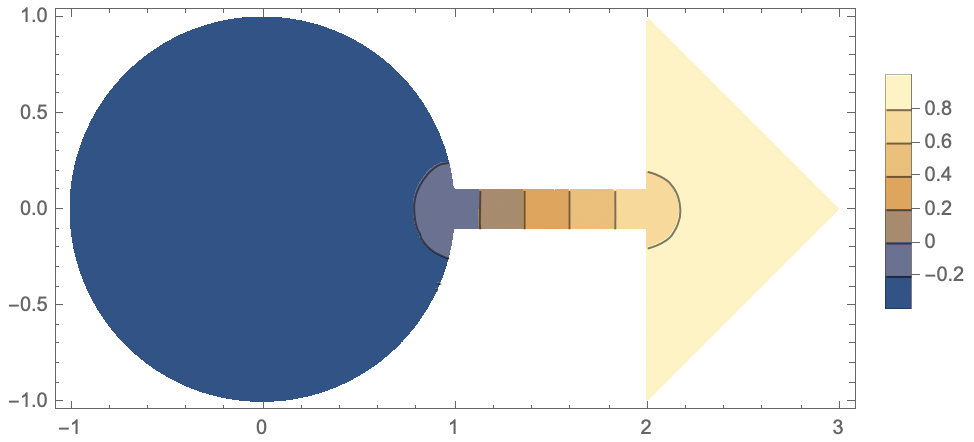}
\caption{First nodal set does not enter the ``deep blue" or ``pale yellow" region.}
\end{figure}
We present a formal statement of the aforementioned result as Theorem \ref{thm: main result 2} after the introduction of various notations in Section \ref{sec: Neumann dumbbells}. The above result also stands in contrast to the Dirichlet case studied in \cite{MS1}, where Mukherjee and the second named author proved that under certain size restrictions on the base domains, the first nodal set is located deep inside one of the base domains and away from the connectors.  It is important to note here that considering mirror couplings on symmetric dumbbells, Theorem 2.1 of \cite{AtarBurdzy2002} proves that the nodal line must intersect the connector. However, the result does not rule out the possibility that the nodal line might be a curve joining the farthest points of the dumbbell through the connector. In this sense, Theorem \ref{thm: main result 2} provides a relatively more precise location for the nodal line on a broader class of dumbbells. 

\section{Construction of the dumbbell and asymptotic localization}\label{sec: construction}

We will work with the one-parameter family of dumbbells as constructed in \cite{Ji} (also see \cite{MS1}). For the sake of completeness, we repeat the construction below.

Consider two bounded disjoint open sets $\Omega_1$ and $\Omega_2$ in $\RR^n, n\geq 2$ with smooth boundary 
such that a portion of the boundary is flat. More precisely, for some positive constant $\xi>0$,

\begin{equation*}
     \overline{\Omega}_1 \cap \{(x_1, x')\in \RR\times \RR^{n-1}: x_1\geq -1; |x'|<3\xi \}= \{(-1, x')\in \pa \Omega_1: |x'|<3\xi \},
\end{equation*}
and 
\begin{equation*}
    \overline{\Omega}_2 \cap \{(x_1, x')\in \RR\times \RR^{n-1}: x_1\leq 1; |x'|<3\xi \}= \{(1, x')\in \pa \Omega_2: |x'|<3\xi \}. 
\end{equation*}
We refer to these domains as the base domains of the dumbbell. Let $Q$ be a line segment joining the flat segments (as described above) of $\pa\Omega_1$ and $\pa\Omega_2$. For some small enough fixed $\epsilon>0$, consider the dumbbell domain $\Omega_\epsilon$  obtained by joining $\Omega_1$ and $\Omega_2$ with a connector $Q_\epsilon$ denoted by
$$\Omega_\epsilon:= \Omega_1\cup \Omega_2 \cup Q_\epsilon.$$
Here
$$Q_\epsilon= Q_1(\epsilon)\cup L(\epsilon)\cup Q_2(\epsilon)$$
is given by
\begin{align*}
     Q_1(\epsilon) &= \left\{(x_1, x')\in \RR \times \RR^{n-1}:  -1 \leq x_1 \leq -1+2\epsilon; |x'|<\epsilon \rho\left(\frac{-1-x_1}{\epsilon} \right) \right\},\\
    Q_2(\epsilon)  &= \left\{(x_1, x')\in \RR \times \RR^{n-1}:  1-2\epsilon \leq x_1 \leq 1; |x'|<\epsilon \rho\left(\frac{x_1-1}{\epsilon} \right) \right\},\\
     L(\epsilon) &=  \left\{(x_1, x')\in \RR \times \RR^{n-1}:  -1+2\epsilon \leq x_1 \leq 1-2\epsilon ; |x'|<\epsilon  \right\},
\end{align*}
where $\rho\in C^\infty((-2, 0))\cap C^0((-2, 0])$ is a positive bump function satisfying
$\rho(0)=2$ and $\rho(q)= 1$ for  $q\in(-2, -1)$.
 \begin{figure}[ht]
\centering
\includegraphics[height=4cm]{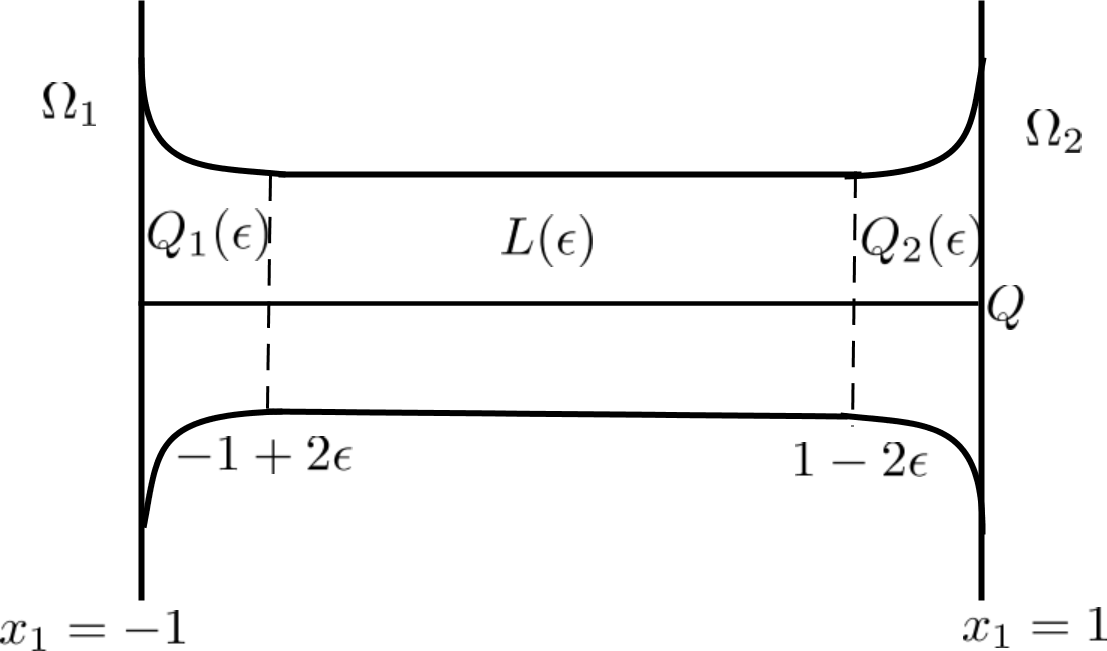}
\caption{Construction of the connector $Q_\epsilon$}
\end{figure}

\subsection{Localization  on Dirichlet Dumbbells}\label{subsec: Dir loc}
Considering the Dirichlet boundary condition on $\Omega_\epsilon$, as pointed out in Chapter 7 of \cite{GN} (see also Section 7 of \cite{Daners} and Chapter 2 of \cite{He}), we have that
$$\lambda_k(\Omega_\epsilon)\to \lambda_k(\Omega_1\cup \Omega_2) \quad \text{ as } \epsilon\to 0.$$
Here, $\lambda_k(\Omega_1\cup \Omega_2)$ denotes the $k$-th element after rearranging the Dirichlet eigenvalues of $\Omega_1$ and $\Omega_2$ non-decreasingly.  Let $\Lambda_i$  denote the spectrum of $\Omega_i$ ($i=1, 2$), and $\varphi_k^{\Omega_i}$ denote the $k$-th $L^2$-normalized Dirichlet eigenfunction of $\Omega_i$ corresponding to the eigenvalue $\lambda_k^{\Omega_i}$. If $\Lambda_1\cap \Lambda_2= \emptyset$ then each $L^2$-normalized eigenfunction $\varphi_{k,\epsilon}$ of the domain $\Omega_\epsilon$ is fully localized in one of the two base domain $\Omega_i$ as $\epsilon\to 0$.  In other words, for every $k\geq 1$, there exists $i\in\{1, 2\}$  such that given any $\delta\in(0,1)$ one can find $\epsilon_0>0$ for which 
$$\|\varphi_{k, \epsilon}\|_{L^2(\Omega_i)}>(1-\delta)~~~\text{ for any  } \epsilon<\epsilon_0.$$
In fact, $\varphi_{k,\epsilon}$ approaches in $L^2$-norm to an eigenfunction $\varphi_{k,0} = \varphi_{k'}^{\Omega_i}$ (for some $i=1,2$ and $k'\leq k$) of the limiting domain which is fully localized in $\Omega_i$ and zero in the other. 

Heuristically, low-energy Dirichlet eigenfunctions cannot tunnel effectively through narrow (sub-wavelength) openings of a domain which leads to their localization in relatively thick parts of the domains. The fact that the spectra of $\Omega_1$ and $\Omega_2$ do not intersect is important for the localization of the eigenfunctions to exactly one subdomain $\Omega_i$.

\subsection{Localization on Neumann Dumbbells}\label{subsec: Neu loc}

The situation is a bit more complicated in the Neumann case. Unlike the Dirichlet case, we cannot assume the Neumann spectra of  $\Omega_1$ and $\Omega_2$ to be disjoint because both spectra contain 0 as their eigenvalue. Since we are interested in studying the second eigenfunction of the dumbbells (as we will see below), this distinction from the Dirichlet case is crucial. Moreover, contrary to their Dirichlet counterpart, the  Neumann eigenvalues and eigenfunctions of the dumbbells may converge to the Dirichlet spectrum of the connecting line segment $Q$ as well. All of these complexities contribute to destroying the asymptotic localization in certain cases.

Let $\psi_k^{\Omega_i}$ ($i=1, 2$) denote the $L^2$-normalized eigenfunction  corresponding to the $k$-th Neumann eigenvalue $\mu_k^{\Omega_i}$  of the base domains $\Omega_i$ ($i=1, 2$). Denote $\psi_{k,\epsilon}$ to be the $L^2$-normalized eigenfunction corresponding to the $k$-th Neumann eigenvalue $\mu_{k,\epsilon}$ of $\Omega_\epsilon$. We assume that $\eta_k$ is the Dirichlet $k$-th eigenvalue of the line segment $Q$.  Defining $\Sigma_i:=\{\mu_l^{\Omega_i}\}$ and $\Lambda_Q:=\{\eta_l\}$, for any eigenvalue $\mu_{k,\epsilon}$,  we have (see \cite{Ji, Ar}) 
\begin{equation}
    \mu_{k,\epsilon}\to \mu \in \Sigma_1 \cup \Sigma_2 \cup \Lambda_Q~~~ \text{as}~~~ \epsilon\to 0.
\end{equation}

We now restate Theorem 1.1 of \cite{Gad}, which deals with the localization of any eigenfunction $\psi_{k,\epsilon}$. We define the multiplicity of $\mu$ as the triple $(n_1,n_2,n_3)$, where $n_1$ and $n_2$ are the multiplicities of $\mu$ in $\Sigma_1$ and $\Sigma_2$ respectively, and $n_3$  is the multiplicity in $\Lambda_Q$. For the ease of restatement and our purpose, we assume $\Omega_1$ and $\Omega_2$ have simple spectra, i.e., $n_1, n_2, $ can be at most 1. The original statement in \cite{Gad} relaxes the simplicity assumption on the spectrum of $\Omega_1$ and $\Omega_2$ and presents a more general description of the following theorem. 

\begin{theorem}[Gadyl'shin]\label{thm: Gadyl'shin}
Assuming that the multiplicity of $\mu$ is $(1,1,0)$, let $\psi_1$ and $\psi_2$ denote the Neumann eigenfunction corresponding to $\mu\in \Sigma_1$ and $\Sigma_2$ respectively. 
Then we have the following convergence:
\begin{align}\label{eq: Neumann loc}
     \left\|\psi_{k,\epsilon}-\alpha_1\psi_1\right\|_{L^2(\Omega_1)} + \left\|\psi_{k,\epsilon}-\alpha_2\psi_2\right\|_{L^2(\Omega_2)} 
     + \|\psi_{k,\epsilon}\|_{L^2(Q_{\epsilon})} \to 0\;\text{as} \;\epsilon\to 0,
\end{align}
where $\alpha_1^2+\alpha_2^2=1$. Moreover, $\psi_{k,\epsilon}$ $L^2$-localizes on $\Omega_1$ whenever $n_2=n_3=0$. More precisely, 
$$\left\|\psi_{k,\epsilon}-\psi_1\right\|_{L^2(\Omega_1)}+\left\|\psi_{k,\epsilon}\right\|_{L^2(\Omega_2)} + \left\|\psi_{k,\epsilon}\right\|_{L^2(Q_\epsilon)}\to 0.$$

\end{theorem}

Since we are interested in the second eigenfunction, we note that the second Neumann eigenvalue $\mu_{2,\epsilon}\searrow 0$ as $\epsilon\to 0$.  In terms of the notation described above, the multiplicity of $0$ (without any assumption on the simplicity of the spectra of $\Omega_1, \Omega_2$) is $(1,1,0)$. Moreover, since $\mu_{1, \epsilon}=0$ (for every $\epsilon$) and $\mu_{3, \epsilon}$ converges to a non-zero value, it is easy to see that for sufficiently small $\epsilon$, the second eigenvalue $\mu_{2,\epsilon}$ is simple. From Theorem \ref{thm: Gadyl'shin}, we have that the $L^2$-mass of any eigenfunction corresponding to $\mu_{2,\epsilon}$ gets distributed across the dumbbell in the limit, 
i.e., $\psi_{2,\epsilon}$ satisfies  (\ref{eq: Neumann loc}). 
 In the particular case of the second eigenfunctions, from Proposition 7.3  of \cite{Ve} (and Proposition 3.6 of \cite{HaVe}), each coefficient in (\ref{eq: Neumann loc}) can be explicitly expressed as 
\begin{equation}\label{eq: coeff}
    \alpha_1 = -\left(\frac{|\Omega_2|}{|\Omega_1| + |\Omega_2|}\right)^{1/2};~~~~~~~~~ \alpha_2 = \left(\frac{|\Omega_1|}{|\Omega_1| + |\Omega_2|}\right)^{1/2}.
\end{equation}

\section{Location of hot spots in a dumbbell with Dirichlet boundary} \label{sec: obstactles}

Consider the one-parameter family of dumbbells $\Omega_\epsilon\subset \RR^n$ as described in Section \ref{sec: construction} with Dirichlet boundary. Fix $\epsilon_0>0$.  
Assume that $\lambda_1^{\Omega_1}\neq \lambda_1^{\Omega_2}$, and without loss of generality, let $\lambda_1^{\Omega_1}< \lambda_1^{\Omega_2}$ with $\varphi_{1,\epsilon}, \varphi_1^{\Omega_1}>0$.  Consider a sequence $\{\epsilon_i\}\searrow 0$ and redefine $\varphi_{1, \epsilon_i}, \varphi_1^{\Omega_1}$ on $\RR^n$ as
\begin{equation*}
    \varphi_{1, \epsilon_i} = \begin{cases}
        & \varphi_{1, \epsilon_i}  \quad \text{ on } \Omega_{\epsilon_i}, \\
        & 0, \qquad \text{otherwise}.
    \end{cases}
    \quad \text{and} \quad 
    \varphi_{1,0} = \begin{cases}
        & \varphi_1^{\Omega_1} \quad \text{ on } \Omega_1,\\
        & 0, \qquad \text{ otherwise}.
    \end{cases}
\end{equation*}
Then,  from the localization results in Subsection \ref{subsec: Dir loc},  we have that $\varphi_{1,\epsilon}$ localizes on $\Omega_1$ and
$$\|\varphi_{1,\epsilon_i}-\varphi_{1,0}\|_{L^2(\RR^n)}\to 0 ~~~\text{as}~~~ \epsilon_i\to 0.$$

\begin{lemma}\label{lem: Dir sup nor conv}
    $\varphi_{1,\epsilon_i}\to \varphi_{1,0}$ as $\epsilon_i\to 0$ in $L^\infty(\Omega_1)$.
\end{lemma}
\begin{proof}
    Consider the equations
$$(\Delta+\lambda_1^{\Omega_1})[\varphi_{1,\epsilon_i}-\varphi_{1,0}]=(\lambda_1^{\Omega_1}-\lambda_{1,\epsilon_i})\varphi_{1,\epsilon_i},$$
and 
$$(\Delta+\lambda_{1,\epsilon_i})\varphi_{1,\epsilon_i}=0$$
on $\Omega_{\epsilon_i}$. Note that $\varphi_{1,0}= 0$ and $\varphi_{1,\epsilon_i}=0$ on $\pa \Omega_{\epsilon_i}$. Also, using domain monotonicity, we have $\lambda_1^{\Omega_1}>\lambda_{1, \epsilon}$ for every $\epsilon\leq \epsilon_0$. Then, applying Theorem 8.15 of \cite{GT} in the above two equations consecutively, 
we have that for some $q>2$ and $\nu\geq\sqrt{\lambda_1^{\Omega_1}}$, 
\begin{align*}
    \|\varphi_{1,\epsilon_i}-\varphi_{1,0}\|_{L^\infty(\Omega_1)}  &\leq \|\varphi_{1,\epsilon_i}-\varphi_{1,0}\|_{L^\infty(\Omega_{\epsilon_i})}\\
    &\leq  C\left(\|\varphi_{1,\epsilon_i}-\varphi_{1,0}\|_{L^2(\Omega_{\epsilon_i})}+ \|(\lambda_{1,0}-\lambda_{1,\epsilon_i})\varphi_{1,\epsilon_i} \|_{L^{q/2}(\Omega_{\epsilon_i})}\right)\\
    &\leq  C\left(\|\varphi_{1,\epsilon_i}-\varphi_{1,0}\|_{L^2(\Omega_{\epsilon_i})}+ C^*|(\lambda_{1,0}-\lambda_{1,\epsilon_i})|\cdot\|\varphi_{1,\epsilon_i} \|_{L^{\infty}(\Omega_{\epsilon_i})}\right)\\
    &\leq  C\left(\|\varphi_{1,\epsilon_i}-\varphi_{1,0}\|_{L^2(\Omega_{\epsilon_i})}+ C'|\lambda_{1,0}-\lambda_{1,\epsilon_i}|\cdot\|\varphi_{1,\epsilon_i} \|_{L^{2}(\Omega_{\epsilon_i})}\right)\\
     &\leq  C\left(\|\varphi_{1,\epsilon_i}-\varphi_{1,0}\|_{L^2(\RR^n)}+ C'|\lambda_{1,0}-\lambda_{1,\epsilon_i}|\cdot\|\varphi_{1,\epsilon_i} \|_{L^{2}(\RR^n)}\right),
\end{align*}
where $C, C'$ depends on $n, q, \nu,$ and $|\Omega_{\epsilon_i}|$. For each $i$, $|\Omega_{\epsilon_i}|$ is uniformly bounded which implies that the constants on the right are independent of $i$. Now using $\lambda_{1,\epsilon_i} \to \lambda_{1,0}$ and $\|\varphi_{1,\epsilon_i}-\varphi_{1,0}\|_{L^2(\RR^n)}\to 0$ we have 
\begin{equation*}\label{eq: sup_norm_conv}
    \|\varphi_{1,\epsilon_i}-\varphi_{1,0}\|_{L^\infty(\Omega_1)}  \to 0 ~~~\text{as}~~~ i\to \infty.
\end{equation*}

\end{proof}

\begin{remark}\label{rem: 3.2}
    Using a similar computation with some routine modifications, it is easy to see that $\varphi_{1,\epsilon_i}\to 0$ in $C^0(\Omega_2)$ as $\epsilon_i\to 0$. 
\end{remark}

For the remainder of the article, we denote the supremum level sets as
$$\displaystyle \LLL(\Omega_1):= \{x\in \Omega_1: \varphi_{1,0}(x)=\|\varphi_{1,0}\|_{L^\infty(\Omega_1)}\},$$ and 
$$\LLL(\Omega_{\epsilon_i}):= \{x\in \Omega_{\epsilon_i}: \varphi_{1,\epsilon_i}(x)=\|\varphi_{1,\epsilon_i}\|_{L^\infty(\Omega_{\epsilon_i})}\}.$$

\begin{remark}
    The maximum of the ground state Dirichlet eigenfunction is always attained in the interior of the domain $\Omega$. So, there are two options: $\mathcal{L}(\Omega)$ can either be a collection of closed curves or a collection of isolated points. Because of the maximum principle, the level set $\mathcal{L}(\Omega)$ cannot be a closed curve. 
\end{remark}

\begin{proof}[Proof of Theorem \ref{thm: main result}]
Consider a sequence of points $x_i\in \LLL(\Omega_{\epsilon_i})$. First, we prove that if $x_0$ is a limit point of $\{x_i\}$, then $x_0\in \LLL(\Omega_1)$.

We begin with the claim that for sufficiently small $\epsilon_i$, $x_i$ is eventually contained in $\Omega_1$. From Remark \ref{rem: 3.2}, given any $\delta>0$, there exists $\epsilon_0>0$ such that $\displaystyle\sup_{x\in \Omega_2} \varphi_{1,\epsilon_i}(x)<\delta \text{ for every } \epsilon_i<\epsilon'.$
If possible, let there exist a subsequence $\{x_j\}\subset \{x_i\}$ such that $x_j\in \Omega_2$ for all $j$. Then 
\begin{equation}\label{eq: Lemma 3.4 eq 1}
    \sup_{x\in \Omega_1} \varphi_{1,\epsilon_j}(x)\leq  \sup_{x\in \Omega_2} \varphi_{1,\epsilon_j}(x)<\delta \text{ for every } \epsilon_j<\epsilon'.
\end{equation}
Now, using Lemma \ref{lem: Dir sup nor conv}, there exists $\epsilon''>0$ such that 
\begin{equation}\label{eq: Lemma 3.4 eq 2}
    \|\varphi_1^{\Omega_1}\|_{\infty}-\delta <\sup_{x\in \Omega_1} \varphi_{1,\epsilon_i}(x) \text{ for every } \epsilon_i<\epsilon''.
\end{equation}
Clearly, choosing $\delta<< \|\varphi_1^{\Omega_1}\|_{\infty}/2$, for any $\epsilon_j<\min\{\epsilon', \epsilon''\}$, (\ref{eq: Lemma 3.4 eq 1}) contradicts $(\ref{eq: Lemma 3.4 eq 2})$. This gives us that for sufficiently small $\epsilon_i$, $x_i\notin \Omega_2$. 

The maximum points $x_i\in \LLL(\Omega_{\epsilon_i})$ cannot lie inside the connector for sufficiently small $\epsilon_i$. In general, if $X$ denotes a max-point of the ground state eigenfunction of a given domain, then the probability that a Brownian motion starting at $X$ escapes the domain within time $t$ can be bounded from below and above, which in turn implies that any max-point must have a (relatively) large ball around it embedded inside the domain. This was formally proved by Georgiev and Mukherjee in \cite{GM2} and as a consequence, it was shown that the radius of such a ball is comparable to  $1/\lambda_1^{\alpha(n)}$, where $\alpha(n)=\frac{1}{4}(n-1)+\frac{1}{2n}$ and $\lambda_1$ denotes the first eigenvalue of the domain. Recently, this inner radius estimate was further improved to  $1/\sqrt{\lambda_{1} (\log \lambda_{1})^{n-2}}$ by Charron and Mangoubi in \cite{CM}. 

Using these estimates, for any $x_i\in \LLL(\Omega_{\epsilon_i})$ there exists a ball of radius comparable to  $1/\sqrt{\lambda_{1,\epsilon_i} (\log \lambda_{1,\epsilon_i})^{n-2}}$ centred at $x_i$ completely contained inside $\Omega_{\epsilon_i}$. Since $\lambda_{1,\epsilon}\to \lambda_1^{\Omega_1}$, $\lambda_{1, \epsilon}$ is bounded from below for sufficiently narrow connector, which in turn implies that $x_i\notin Q_{\epsilon_i}$ for sufficiently small $\epsilon_i$. This proves that $x_i$ is eventually contained in $\Omega_1$ for sufficiently large $i$.

Recall that given $\{x_i\}\to x_0$, we want to show that $\displaystyle \varphi_{1,0}(x_0)=\|\varphi_{1,0}\|_{L^\infty(\Omega_1)}$. Observe that
\begin{align}
 \nonumber   |\varphi_{1,\epsilon_i}(x_i)-\varphi_{1,0}(x_0)|&= |\varphi_{1,\epsilon_i}(x_i)- \varphi_{1,\epsilon_i}(x_0)+\varphi_{1,\epsilon_i}(x_0) - \varphi_{1,0}(x_0)|\\
 \nonumber &\leq |\varphi_{1,\epsilon_i}(x_i)- \varphi_{1,\epsilon_i}(x_0)| + |\varphi_{1,\epsilon_i}(x_0) - \varphi_{1,0}(x_0)|.
 \end{align}
Since $\{x_i\}\to x_0$, we have that $\varphi_{1,\epsilon_i}(x_i)\to \varphi_{1,\epsilon_i}(x_0)$. Moreover, using (\ref{eq: sup_norm_conv}), we have
\begin{equation}
    |\varphi_{1,\epsilon_i}(x_0) - \varphi_{1,0}(x_0)|\leq \|\varphi_{1,\epsilon_i}-\varphi_{1,0}\|_{L^\infty(\Omega_1)}\to 0 \quad \text{as} \quad i\to \infty,
\end{equation}
which gives us
\begin{equation}
    |\varphi_{1,\epsilon_i}(x_i)-\varphi_{1,0}(x_0)|\to 0 \quad \text{as} \quad i\to \infty.
\end{equation}
Since $\varphi_{1,\epsilon_i}(x_i)=\|\varphi_{1, \epsilon_i}\|_{L^\infty(\Omega_1)}\to \|\varphi_{1, 0}\|_{L^\infty(\Omega_1)}$, we have  $x_0\in \LLL(\Omega_1)$.

Note that so far 
we have not imposed any kind of topological restrictions on the base domains of the dumbbells. 
Also, recall that in the hypothesis of Theorem \ref{thm: main result}, we have assumed  $\Omega_1$ to be convex and that $\lambda_1(\Omega_1)<\lambda_1(\Omega_2)$ to ensure the localization on $\Omega_1$. Since $\Omega_1$ is convex, we have that the ground state eigenfunction has a unique maximum point. This follows from the fact that the ground state eigenfunction is log-concave \cite{BL}. Let $x_0$ denote the max point of the ground state eigenfunction of $\Omega_1$. Also, let $B(x_0, \delta)$ denote a ball of radius $\delta$ centred at $x_0$.

Finally, we show that given any $\delta>0$, 
there exists $\epsilon'>0$ such that 
$$\LLL(\Omega_\epsilon) \subset B(x_0, \delta)~~~ \text{for any}~~~ \epsilon\leq \epsilon'.$$ 
If not, then for some $\delta>0$, there exist a sequence $\{x_j\}\subset \Omega_1$ such that $x_j\in \LLL(\Omega_{\epsilon_j})$ for each $i$, and $x_i\notin B(x_0, \delta)$. Then there is a subsequence $\{x_k\}\subset \{x_j\}$ such that $x_k\to x'$ for some $x'\in \Omega_1\setminus B(x_0, \delta)$. 
From the above argument, we have that $x'\in \LLL(\Omega_1)$. Since $\Omega_1$ has exactly one point of maximum, this gives a contradiction. 
\end{proof}

Now, let's briefly shift our focus to a result by Georgiev and Mukherjee from \cite{GM} which discusses the position of an obstacle inside a domain that satisfies a certain asymmetry condition (defined below).

\begin{defi}[$\alpha-$asymmetry]
    A bounded domain $\Omega \subset \RR^n$ is said to be $\alpha-$asymmetric if for all $x \in \partial \Omega$, and all $r > 0$,
$$\frac{\Vol(B(x,r)\setminus \Omega_{\epsilon})}{\Vol(B(x,r))} \geq \alpha.$$
\end{defi}
The above notion of asymmetry was introduced in \cite{Hay} and basically rules out ``sharp spikes'' (spikes with relatively small volume) entering deep in $\Omega$. Note that convex domains trivially satisfy the above asymmetry condition with coefficient $\alpha=\frac{1}{2}$. 

\begin{theorem}\cite[Theorem 4.1]{GM}\label{thm: GM result}
Let $\Omega\subset \RR^n$ be a domain that satisfies the $\alpha$-asymmetric condition for some $\alpha$ and $D$ be a convex obstacle. Moreover, assume that  $y_0+D$ maximizes $\lambda_1(\Omega\setminus \overline{(y+D)})$. Then there exists a constant $C_0(\alpha, n)$ such that if $\lambda_1(\Omega\setminus \overline{(y_0+D)})>C\lambda_1(\Omega)$ for some $C\geq C_0$, then  $\dist(\LLL(\Omega), y_0+D)=0$. In other words, if the obstacle is sufficiently large, then $\dist(\LLL(\Omega), y_0+D)=0$. 
\end{theorem}

The proof of Theorem \ref{thm: GM result} relies heavily on the fact  
(proved in \cite{GM2}) that, given any max point of the ground state eigenfunction, one can ``almost inscribe'' a wavelength scale ball centered at the max point, and under the asymmetry condition, the ball is fully inscribed with radius $r_0/\sqrt{\lambda_1}$ where $r_0$ depends on $\alpha$ and $n$. 
We refer our readers to \cite{GM2, GM} for complete details. 

Considering dumbbells $\Omega_\epsilon$, recall that $\Omega_\epsilon:=\Omega_1\cup Q_\epsilon \cup \Omega_2$, where $\Omega_1$ is convex. Observe that $\Omega_1\cup Q_\epsilon$ is $\frac{1}{4}-$asymmetric (follows from the convexity of $\Omega_1$ and the concavity of $Q_1(\epsilon), Q_2(\epsilon)$ in $Q_\epsilon$). From the proof of Theorem \ref{thm: main result}, we know that for sufficiently narrow connectors, the set of max points of the dumbbell $\LLL(\Omega_\epsilon)$ is completely contained inside $\Omega_1$. Then such a fully inscribed ball (described above) centered at $x_\epsilon\in \LLL(\Omega_\epsilon)$ 
should be contained inside $\Omega_1$. So, when dealing with dumbbell domains, it is enough to consider the $\alpha-$asymmetry on $\Omega_1\cup Q_\epsilon$. In other words, even if $\Omega_2$ has ``sharp spikes'', the spikes do not enter $\Omega_1$, and hence it does not affect the estimates for Theorem \ref{thm: GM result}. This tells us that we can relax the asymmetry condition on $\Omega_2$, and Theorem \ref{thm: GM result} is true for any dumbbell as described in Theorem \ref{thm: main result}. 

\begin{proof}[Proof of Corollary \ref{cor: obstacle}]
Let $x_0$ denote the unique maximum point of the ground state eigenfunction of $\Omega_1$ and $y_\epsilon+D$ be the optimal location of the obstacle $D$ inside the dumbbell $\Omega_\epsilon$. From Theorem \ref{thm: main result} we have that, given any $\delta>0$ there exists $\epsilon'>0$ such that $\LLL(\Omega_{\epsilon}) \subset B(x_0,\delta)$ for every $\epsilon<\epsilon'$. If the obstacle is large enough (in the sense as described in Theorem \ref{thm: GM result}), we have from Theorem \ref{thm: GM result} that for some ground state max point $x_\epsilon$ of $\Omega_\epsilon$, $x_\epsilon \in y_\epsilon+D$. Combining these two facts, we have that $B(x_0, \delta)\cap (y_\epsilon+D)\neq \emptyset$, which in turn implies that $\dist(x_0, y_\epsilon+D)<\delta$ for every $\epsilon<\epsilon'$.
\end{proof}

We end this section with the following proposition\footnote{The proposition was originally developed by Mayukh Mukherjee and the second named author in an earlier version of \cite{MS1}. The authors of the present work requested Mukherjee for joint authorship, but he declined and conveyed that there were no conflicts of interest. Nevertheless, the present authors would like to acknowledge him and thank him for his contribution.} where we address a mass concentration question in the connector of the dumbbell, following the general line of inquiry in \cite{BD, vdBB}. \cite{GM3} also establishes $L^\infty$-estimates in domains with long narrow tubes. 

\begin{proposition}\label{prop: $L^2$-norm_decay_connector}
    Consider a dumbbell domain $\Omega_\epsilon \subset \RR^n$, where $n \geq 2$ (see diagram below). Let $C(z) := \Omega_\epsilon \cap \{x\in \RR^n: x_1=z\}$ be the cross-section of $\Omega_\epsilon$ at $x_1=z\in \RR$ by a hyperplane perpendicular to the coordinate axis $x_1$ and let $\mu(z)$ be the first eigenvalue of the Laplace operator in $C(z)$, with Dirichlet boundary condition on $\pa C(z)$ and $\mu= \inf\{\mu(z): z\in (z_0, z_2)\}$. If $\lambda_{1,0}< \mu$ then, there exists a constant $D$ (depending on $\Omega_1, \Omega_2$) such that
    \begin{equation*}\label{eq: $L^2$-norm_decay_connector}
        \|\varphi_{1,\epsilon}\|_{L^2(Q_\epsilon)}^2  \leq D \|\varphi_{1,\epsilon}\|_{L^2(Q(0))}^2.
    \end{equation*}
\end{proposition}


\begin{proof} We use the following result from \cite{DNG}:

\begin{theorem}[Delitsyn, Nguyen and Grebenkov]\label{thm: DNG main theorem} 
Let $z_1= \inf\{z\in \RR: C(z)\neq \emptyset\}, \hspace{6mm} z_2= \sup\{z\in \RR: C(z)\neq \emptyset\}$. Fix $z_0\in (z_1, z_2)$. 
Let $\varphi$ be a Dirichlet-Laplacian eigenfunction in $\Omega$, and $\lambda$ the associated eigenvalue. If $\lambda< \mu$, then 
\begin{equation}\label{ineq: DNG main theorem}
    \|\varphi\|_{L^2(C(z))} \leq \|\varphi\|_{L^2(C(z_0))}e^{-\beta \sqrt{\mu-\lambda}(z-z_0)}, z>z_0.
\end{equation}
with $\beta= 1/\sqrt{2}$.
\end{theorem}

Without loss of generality, assume that $\varphi_{1,\epsilon}$ is localized on $\Omega_1$ as $\epsilon\to 0$. Following the terminology in \cite{DNG}, choosing $z_0=0$, we refer to $\Omega_1$ as the ``basic" domain and $\Omega_\epsilon\setminus \Omega_1$ the ``branch". In our case of dumbbells,  $\mu$ is attained for some $z$ for which $C(z)\in \Omega_2$. 
 \begin{figure}[ht]
\centering
\includegraphics[height=3.5cm]{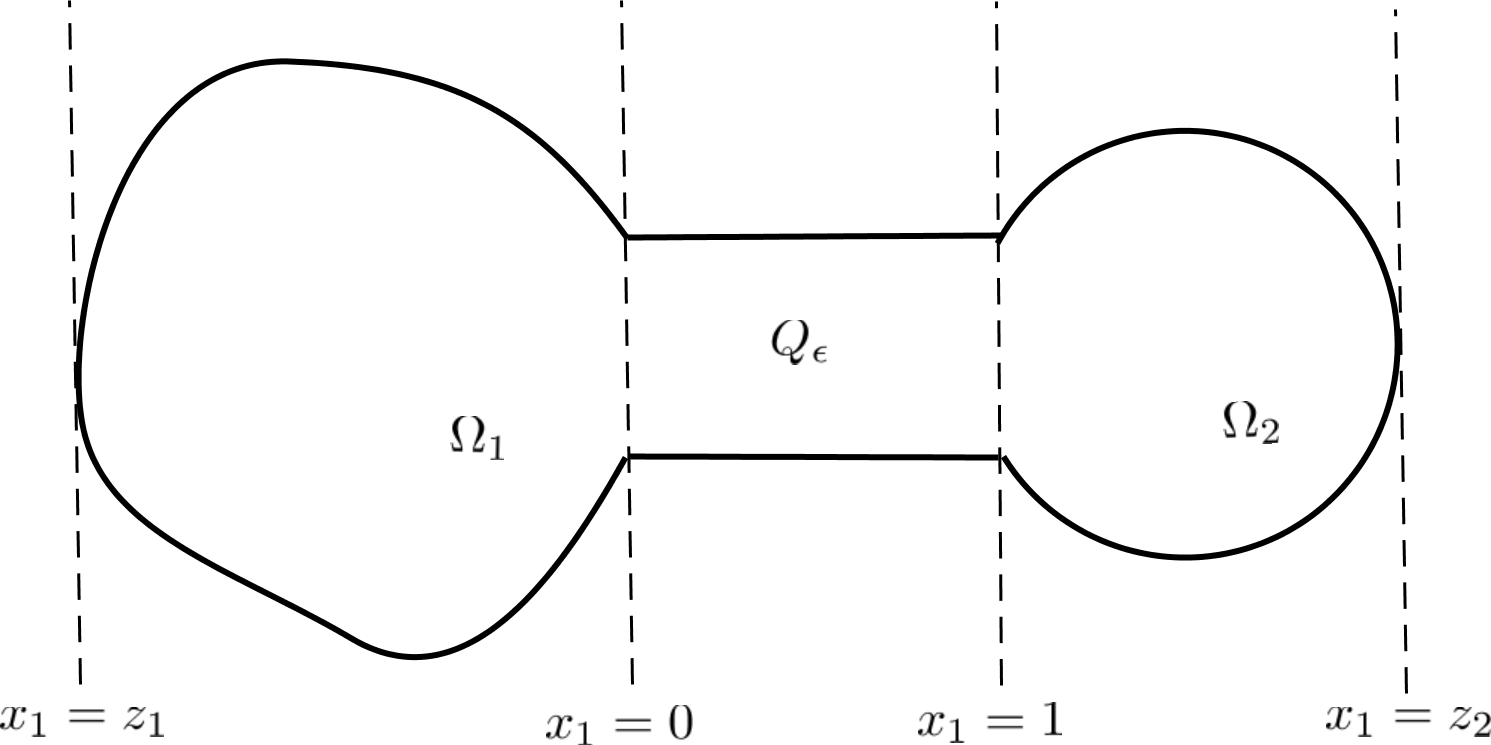}
\caption{Mass concentration in the  connector}
\end{figure}

Note that if $\epsilon'< \epsilon$, then $\lambda_{1,\epsilon'} > \lambda_{1,\epsilon}$, which implies $\lambda_{1,\epsilon}$ monotonically converges to $\lambda_{1,0}$. From the assumption that $\lambda_{1,0}<\mu$ (this assumption puts a restriction on the ``fatness'' of the base domains of the dumbbell), for small enough $\epsilon$, $\lambda_{1,\epsilon}< \mu$. 


Now, applying Theorem \ref{thm: DNG main theorem} for $\varphi_{1,\epsilon}$, we have
\begin{equation*}
    \|\varphi_{1,\epsilon}\|_{L^2(C(z))} \leq \|\varphi_{1,\epsilon}\|_{L^2(C(0))} e^{-\beta (\sqrt{\mu-\lambda_{1,\epsilon}})z}, z>0.
\end{equation*}
The assumption $\displaystyle\lambda_{1,0}< \mu$, together with the monotonicity $\lambda_{1,\epsilon}< \lambda_{1,0}$ implies $\mu- \lambda_{1,\epsilon} > \mu - \lambda_{1,0}$. Using this, we can rewrite the above inequality as
\begin{equation}
    \|\varphi_{1,\epsilon}\|_{L^2(C(z))}^2 \leq \|\varphi_{1,\epsilon}\|_{L^2(C(0))}^2 e^{-2\beta (\sqrt{\mu - \lambda_{1,0}})z}, z>0.
\end{equation}
Integrating both sides of the above inequality from $z=0$ to $z=1$, we have,
\begin{align}
\|\varphi_{1,\epsilon}\|_{L^2(Q_{\epsilon})}^2 = \int_0^1 \|\varphi_{1,\epsilon}\|_{L^2(C(z))}^2 dz &\leq \int_0^1 \|\varphi_{1,\epsilon}\|_{L^2(C(0))}^2 e^{-2\beta (\sqrt{\mu - \lambda_{1,0}})z} dz\\
&\leq D(\Omega_1, \Omega_2)\|\varphi_{1,\epsilon}\|_{L^2(C(0))}^2,
\end{align}
where $\displaystyle D(\Omega_1, \Omega_2)= \frac{1}{ \sqrt{2(\mu - \lambda_{1,0}})}(1- e^{- \sqrt{2(\mu - \lambda_{1,0}})})$.
\end{proof}

\section{Location of the second nodal line in dumbbells with Neumann boundary}\label{sec: Neumann dumbbells}

In this section, we will restrict ourselves to planar domains.  We begin with the following well-known result which gives us a topological description of the nodal set of the first non-constant Neumann eigenfunction.

\begin{proposition}[Pleijel]\label{prop: neumannn conj}
    Let $\Omega\subset\RR^2$ be a planar simply-connected Euclidean domain and $\psi$ be an eigenfunction corresponding to the first non-zero Neumann eigenvalue. Then $\psi$ cannot have a closed nodal curve. In other words, $\NNN(\psi)$ must intersect the boundary at exactly two points.
\end{proposition}
\begin{proof}
    If possible, let the nodal curve $\NNN(\psi)$ be closed and $\Omega'\subset\Omega$ be the nodal domain for which $\pa\Omega=\NNN(\psi)$. Note that $\psi$ satisfies the Helmholtz equation on $\Omega'$ and $\psi|_{\pa\Omega'}=0$, which implies that the first Dirichlet eigenvalue $$\lambda_1(\Omega')=\mu_2(\Omega).$$ Using domain monotonicity, we have $\lambda_1(\Omega')>\lambda_1(\Omega)$. Now, considering the relation between Dirichlet and Neumann eigenvalues, first obtained by P{\'o}lya \cite{pol}, we have 
    $$\mu_2(\Omega) \leq \lambda_1(\Omega),$$
    which contradicts domain monotonicity. This completes the proof.
\end{proof}

Now that we know what the nodal set looks like,  our goal is to give a precise location of the first nodal set of a dumbbell. Considering a one-parameter family of dumbbell domains $\Omega_\epsilon$ (as described in Section \ref{sec: construction}) 
with Neumann boundary, we introduce some more notations that will be used in restating and proving Theorem \ref{thm: informal main result 2}.

\subsection{Notations}\label{subsec: notations}
 Fix $\epsilon_0>0$, and two half-disks $D_ {r_i}\subset\Omega_i$ ($i=1, 2$) of radius $r_i$ centered at the endpoints of the line segment $Q$ such that the half-disk $D_{r_i/2}$ is also contained in $\Omega_i$ and the semicircle $S_{r_i}$ divides the domain $\Omega_\epsilon$ into exactly two connected components with one of the components being $\Omega_i'\subset\Omega_i$ for any $\epsilon\leq \epsilon_0$ (say).   Denote $\Gamma_i$ to be the boundary of $\Omega_i'$  and $\widehat{\Gamma}_i$ the boundary of $\Omega_1\setminus D_{r_i/2}$. Given any  curve $\Gamma\subset\pa\Omega_i$, $T_{\Gamma, \delta}$ denotes a $\delta$-tubular neighbourhood of $\Gamma$. Finally, define $T_{\Gamma, \delta}^-:= T_{\Gamma, \delta} \cap \Omega_i'$ and $T_{\Gamma, \delta}^+:= T_{\Gamma, \delta}\setminus T_{\Gamma, \delta}^-$. 

\begin{figure}[ht]
\centering
\includegraphics[height=4.2cm]{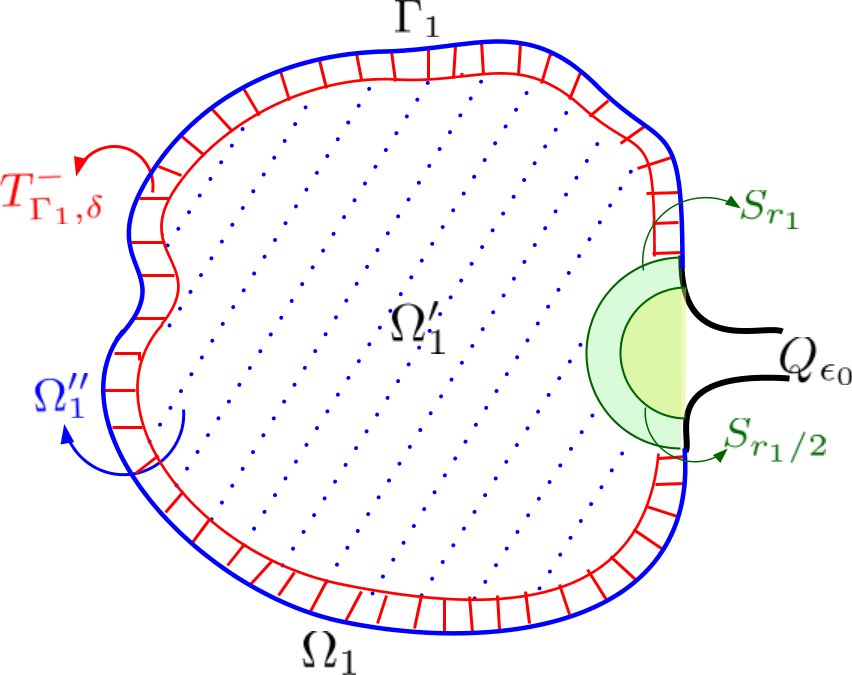}
\caption{Base domain $\Omega_1$ of $\Omega_\epsilon$}
\label{fig: subdomain}
\end{figure}

\begin{theorem}\label{thm: main result 2}
    Let $\Omega_1, \Omega_2\subset\RR^2$ be any two simply connected bounded domains, and $\Omega_{\epsilon}$ be a one-parameter family of dumbbells (as described in Section \ref{sec: construction}) whose connector widths go to zero as $\epsilon\to 0$. Fixing $r_1, r_2>0$ (however small), there exists $\epsilon'>0$ such that the first nodal line of $\Omega_{\epsilon}$ is contained in $D_{r_1}\cup Q_\epsilon\cup D_{r_2}$ for any $\epsilon<\epsilon'$. 
\end{theorem}

From the discussion in Section \ref{sec: construction} we have that \begin{equation}\label{eqn: L^2-convergence}
    \|\psi_{2,\epsilon}-\alpha_1 \psi^{\Omega_1}_1\|_{L^2(\Omega_1)}+\|\psi_{2,\epsilon}-\alpha_2 \psi^{\Omega_2}_1\|_{L^2(\Omega_2)} +\|\psi_{2,\epsilon}\|_{L^2(Q_\epsilon)} \to 0\quad \mbox{as} \; \epsilon \to 0
\end{equation}
where $\psi^{\Omega_1}_1\equiv \frac{1}{\sqrt{|\Omega_1|}},~ \psi^{\Omega_2}_1 \equiv \frac{1}{\sqrt{|\Omega_2|}}$ and  
$\alpha_1, \alpha_2$ are as in (\ref{eq: coeff}).  We redefine $\psi_1^{\Omega_i}= \alpha_i\psi_1^{\Omega_i}$. Note that now the limiting eigenfunctions of the base domains are not $L^2$-normalized. Also, recall that both $\psi_1^{\Omega_1}$ and $\psi_1^{\Omega_2}$ are non-zero constants. The crux of the proof of Theorem \ref{thm: main result 2} lies in the following 

\begin{lemma}
$\|\psi_{2,\epsilon}-\psi_1^{\Omega_i}\|\to 0$ in $L^\infty(\Omega_i')$ where $\Omega_i'$ is defined as in Subsection \ref{subsec: notations}
\end{lemma}

\begin{proof}

Consider two subdomains of $\Omega_1'$, a boundary layer $\Omega^*:=T_{\widehat{\Gamma}_1, 2\delta}^-$ (choice of $\delta$ to be determined later) and 
$\Omega_1'':=\Omega_1'\setminus T_{\Gamma_1, \frac{\delta}{2}}$
which is compactly contained inside $\Omega_1$. Since $T_{\Gamma_1, \delta}^-\cup \Omega_1''= \Omega_1'$, in order to show $L^\infty$-convergence on $\Omega_1'$, it is enough to show the convergence on $\Omega_1''$ and $T_{\Gamma_1, \delta}^-$. 
\subsection*{Convergence outside the boundary layer}

 Considering the following equation
\begin{align*}
    (\Delta + \mu_{2,\epsilon})(\psi^{\Omega_1}_1 - \psi_{2,\epsilon})&= \mu_{2,\epsilon} \psi_1^{\Omega_1} \text{ on } \Omega_1,
\end{align*}
and applying Theorem 8.24 of \cite{GT}, we have that on $\Omega''_1 \Subset \Omega_1$ for any $q>2$ and $\epsilon\leq \epsilon_0$ 
\begin{align*}
    \|\psi^{\Omega_1}_1 - \psi_{2,\epsilon}\|_{L^\infty(\overline{\Omega_1''})} 
    &\leq  C_1\left(\|\psi^{\Omega_1}_1 - \psi_{2,\epsilon}\|_{L^2(\Omega_1)}+ \|\mu_{2,\epsilon} \psi_1^{\Omega_1} \|_{L^{q/2}(\Omega_1)}\right)\\
    &=  C_1\left(\|\psi^{\Omega_1}_1 - \psi_{2,\epsilon}\|_{L^2(\Omega_1)}+ |\mu_{2,\epsilon}| \cdot|\psi_1^{\Omega_1}|\cdot |\Omega_1|^{\frac{2}{q}}\right)\to 0  \quad \text{as } \epsilon \to 0,
\end{align*}
where $C_1=C_1(q,\frac{\delta}{2}, \mu_{2,\epsilon_0}, |\Omega_{\epsilon_0}|)$. 

\subsubsection*{Extension of Neumann eigenfunctions across the boundary}
Note that these Moser-type estimates used above can only be applied to subdomains compactly contained inside $\Omega_1$. If we can extend the Neumann eigenfunction across the boundary and reduce the estimates near the boundary to an interior estimate, then considering that $T_{\Gamma_1, \delta}^-\Subset T_{\widehat{\Gamma}_1, 2\delta}$ we can use Theorem 8.24 of \cite{GT} on $T_{\Gamma_1, \delta}^-$ to obtain our required convergence.


Note that on $\Omega^*$, we have 
\begin{equation}\label{eq: Laplace Euclidean tube}
     -\Delta_{\Omega^*}\psi_{2,\epsilon} =-\frac{\partial^2\psi_{2,\epsilon}}{\partial x^2} - \frac{\partial^2\psi_{2,\epsilon}}{\partial y^2} = \mu_{2,\epsilon}\psi_{2,\epsilon} 
\end{equation}
with outward normal derivative $\frac{\pa\psi_{2,\epsilon}}{\pa \eta}=0$ on $\widehat{\Gamma}_1$. 

Let $\gamma_0:(-a, a) \to \widehat{\Gamma}_1$ be the arc length parametrization of $\widehat{\Gamma}_1$. On the inner $2\delta$-neighbourhood of $\widehat{\Gamma}_1$, we can define a coordinate system $\Phi: \mathcal{R}_{2\delta} \to \Omega^*$ as follows:
\begin{align*}
    (t,s) \to (x,y):= \gamma_0(t) + s\widehat{n}_0(t),
\end{align*}
where $\widehat{n}_0(t)$
is the inward unit normal vector at $\gamma_0(t)$, and 
$$\mathcal{R}_{2\delta}:=\{(t, s)\in \RR^2: |t|<a \text{ and } 0\leq s<2\delta\}.$$
The determinant of the Jacobian of the above transformation is given by
$$\det(D\Phi(t,s)) = \widehat{n}_0(t)\cdot \widehat{n}_0(t) +s \widehat{n}_0(t)\cdot \ddot{\gamma}_0(t) = -1+c(t)s,$$ 
where $c$ is a continuous function on $\widehat{\Gamma}_1$. Since $\overline{\widehat{\Gamma}}_1$ is compact, we choose $\delta$ sufficiently small and independent of $t$ such that $D\Phi(t,s)$ is invertible for all $0<s < 2\delta$ and for all $t\in \widehat{\Gamma}_1$.

In $(t,s)$-coordinates, the Laplace-Beltrami operator near a small neighborhood of the boundary $\{s=0\} \cap \mathcal{R}_{2\delta}$ is
\begin{equation}\label{eqn: Laplace in Fermi}
        -L = - \frac{1}{\|\Dot{\gamma}_s(t)\|^2}\partial_{tt} - \partial_{ss}+\frac{\langle \dot{\gamma}_s(t),\ddot{\gamma}_s(t)\rangle}{\|\Dot{\gamma}_s(t)\|^4} \partial_t  - \frac{\langle \Dot{\gamma}_s(t),\Dot{n}_0(t) \rangle}{\|\dot{\gamma}_s(t)\|^2}\partial_s,
    \end{equation}
where $\gamma_s(t):=\Phi(t, s)$ denotes the layer curve.
Then, under these new coordinates, the Helmholtz equation (\ref{eq: Laplace Euclidean tube}) changes to 
\begin{equation}
    -L\psi'_{2, \epsilon} = \mu_{2, \epsilon} \psi'_{2, \epsilon},
\end{equation}
with the boundary condition $\displaystyle\frac{\partial {\psi_{2,\epsilon}'}}{\partial s} = 0$ on $\overline{\mathcal{R}}_{2\delta}\cap \{s=0\}$ where $\psi_{2,\epsilon}' = \psi_{2,\epsilon} \circ \Phi$.

We rename $\psi'_{2,\epsilon}$ as $\psi_{2,\epsilon}$ for notational convenience. Also, we rename the coefficients of $-L$ as $\displaystyle G_{11}=-\frac{1}{\|\Dot{\gamma}_s(t)\|^2};~~G_{22}= -1;~~ F_1= \frac{\langle \dot{\gamma}_s(t),\ddot{\gamma}_s\rangle}{\|\Dot{\gamma}_s(t)\|^4};~~ F_2= -\frac{\langle \Dot{\gamma}_s(t),\Dot{n}_0(t) \rangle}{\|\dot{\gamma}_s(t)\|^2}$. Following the technique used in \cite{HuShiXu15}, we now extend the coefficients $G_{11},G_{22},F_1$ and the eigenfunction $\psi_{2,\epsilon}$ as even functions in $s$ and the coefficient $F_2$ as odd function in $s$ across $\widehat{\Gamma}_1$ to the domain 
$$\widetilde{\mathcal{R}}_{2\delta}= \{(t, s)\in \RR^n:  |t| < a, |s|<2\delta\}.$$ 
In effect, we are extending the Laplace operator on $\Omega^*$ to an operator on $T_{\widehat{\Gamma}_1, 2\delta}$. We call the above extensions $\widetilde{G}_{ii}$, $\widetilde{F}_i$ ($i=1, 2$), and  obtain the extended operator $-\widetilde{L}$ on $\widetilde{\mathcal{R}}_{2\delta}$ defined as
$$-\widetilde{L}=\widetilde{G}_{11}\pa_{tt}+ \widetilde{G}_{22}\pa_{ss}+ \widetilde{F}_1 \pa_t + \widetilde{F}_2 \pa_s.$$
Then the extended eigenfunction  $\widetilde{\psi}_{2,\epsilon}$ satisfies
$$-\widetilde{L}\widetilde{\psi}_{2,\epsilon}=\mu_{2,\epsilon} \widetilde{\psi}_{2,\epsilon} \quad \text{on $\widetilde{\mathcal{R}}_{2\delta}$}.
$$ 

\begin{remark}
    By choosing appropriate coordinates within a tubular neighborhood of the boundary, the above extension holds in higher dimensions as well. As such, the $L^\infty$-convergence in this lemma and the arguments of Theorem \ref{thm: main result 2} similarly apply to higher dimensions. 
\end{remark}

\subsection*{Convergence on the boundary layer}

From the above computation, $\widetilde{\mathcal{R}}_{2\delta}$ is identified with $T_{\widehat{\Gamma}_1, 2\delta}$ and as mentioned above, we have that $T_{\Gamma_1, \delta}^-\Subset T_{\widehat{\Gamma}_1, 2\delta}$. Moreover, when restricted to $T_{\Gamma_1, \delta}^-$, we have that $\widetilde{L}=L$ and $\widetilde{\psi}_{2,\epsilon}=\psi_{2,\epsilon}$. Using the fact that a $C^2$-function $f$ on $[0,\infty)$ satisfying $f'(0) = 0$ has an even extension which is also $C^2$ on $\RR$, we have that $\widetilde{\psi}_{2,\epsilon}$ is at least $C^2$ on $\widetilde{\mathcal{R}}_{2\delta}$. 
Also note that $\widetilde{L}$ satisfy conditions $(8.5)$ and $(8.6)$ of Chapter 8 in \cite{GT}. We point out that the fact that certain coefficients of $\widetilde{L}$ are discontinuous on $\widetilde{\mathcal{R}}_{2\delta}\cap \{s=0\}$ is not an issue when applying Theorem 8.24 of \cite{GT}. 
Now, consider the equation 
\begin{align*}
    \widetilde{L}(\psi^{\Omega_1}_1 - \widetilde{\psi}_{2,\epsilon})&= \mu_{2,\epsilon} \psi_1^{\Omega_1} \text{ on } \mathcal{\widetilde{R}}_{2\delta}.
\end{align*}
The $L^2$-convergence of eigenfunctions on $T^-_{\Gamma_1, \delta}$ (which follows from Theorem \ref{thm: Gadyl'shin}), implies the $L^2$-convergence of the corresponding extended eigenfunctions on $T_{\widehat{\Gamma}_1, 2\delta}$. Now, applying Theorem 8.24 of \cite{GT}, we have 
\begin{align*}
    \|\psi^{\Omega_1}_1 - \psi_{2,\epsilon}\|_{L^\infty(T_{\Gamma_1, \delta}^-)} 
    &\leq  C_2\left(\|\psi^{\Omega_1}_1 - \widetilde{\psi}_{2,\epsilon}\|_{L^2(\widetilde{\mathcal{R}}_{2\delta})}+ \|\mu_{2,\epsilon} \psi_1^{\Omega_1} \|_{L^{q/2}(\widetilde{\mathcal{R}}_{2\delta})}\right)\\
    &=  C_2\left(\|\psi^{\Omega_1}_1 - \widetilde{\psi}_{2,\epsilon}\|_{L^2(\widetilde{\mathcal{R}}_{2\delta})}+ |\mu_{2,\epsilon}| \cdot|\psi_1^{\Omega_1}|\cdot |\widetilde{\mathcal{R}}_{2\delta}|^{\frac{2}{q}} \right) \to 0  \quad \text{as } \epsilon \to 0,
\end{align*}
where $C_2= C_2(q,\frac{\delta}{2} , \mu_{2,\epsilon_0}, |\Omega_{\epsilon_0}|)$. One can repeat a similar computation on $\Omega_2'$ as well, which gives our required convergence.
\end{proof}
We are now ready to prove that the second nodal line of the Neumann dumbbell cannot enter too deep inside either of the two base domains of a dumbbell with a sufficiently narrow connector. 

\begin{proof}[Proof of Theorem \ref{thm: main result 2}]
Recall that $\psi_1^{\Omega_i}$ is a non-zero constant and hence does not change sign in $\Omega_i$ ($i=1,2$). Moreover, from the values of $\alpha_1$ and $\alpha_2$ in (\ref{eq: coeff}) we have that $\psi_1^{\Omega_1}$ and $\psi_1^{\Omega_2}$ are of opposite signs. 
From Proposition \ref{prop: neumannn conj}, we know that the nodal line $\NNN(\psi_{2,\epsilon})$ intersects the boundary at exactly two points and divides $\Omega_\epsilon$ into two components for every $\epsilon$. 

We prove that $\NNN(\psi_{2,\epsilon})$ does not enter $\Omega_1'$ for sufficiently small $\epsilon$. If possible, let there exist a sequence $\{\epsilon_i\}\to 0$ as $i\to \infty$ such that $\NNN(\psi_{2,\epsilon_i})$ enters $\Omega_1'$  for every $i\in \NN$. Then $\NNN(\psi_{2,\epsilon_i})$ intersects $S_{r_1}$ for every $\epsilon_i$. Let $y_i\in \NNN(\psi_{2,\epsilon_i}) \cap S_{r_1}$. Then there exists and subsequence $\{j\}\subset \{i\}$ such that $y_j\to y$ for some $y\in S_{r_i}$ and $\psi_{2,\epsilon_i}(y_j)=0$. Since $\varphi^{\Omega_1}_1<0$ in $\Omega_1$ and $\displaystyle \|\psi_{2,\epsilon}- \psi_1^{\Omega_1}\|_{C^0(\overline{\Omega_1'})} \to 0$, we have that 
\begin{align*}
    |\psi_1^{\Omega_1}(y)|=\lim_{j\to \infty}|\psi_1^{\Omega_1}(y_j)|=\lim_{j\to \infty}|\psi_1^{\Omega_1}(y_j)|= \lim_{j\to \infty}|(\psi_{2,\epsilon_j}-\psi_1^{\Omega_1})(y_j)|=0,
\end{align*}
a contradiction. Similarly, $\NNN(\psi_{2,\epsilon})$ does not enter $\Omega_2'$ for sufficiently small $\epsilon$. This completes the proof.
\end{proof}

\subsection*{Acknowledgements} 
 The first named author expresses gratitude to the Indian Institute of Technology Bombay for providing a conducive working atmosphere. The research of the first named author was supported by the Prime Minister’s Research Fellowship (PMRF) program (Fellowship no. 1301599).  The second named author would like to thank Iowa State University for funding the research and providing ideal working conditions. The authors would like to acknowledge useful conversations and correspondence with Mayukh Mukherjee, Gabriel Khan, and Xuan Hien Nguyen.

\bibliography{main}

\begin{thebibliography}{AAKK21}

\bibitem[AAKK21]{AAK}
T.~V. Anoop, K.~Ashok~Kumar, and S.~Kesavan.
\newblock A shape variation result via the geometry of eigenfunctions.
\newblock {\em J. Differ. Equations}, 298:430--462, 2021.

\bibitem[AB02]{AtarBurdzy2002}
R.~Atar and K.~Burdzy.
\newblock On nodal lines of {N}eumann eigenfunctions.
\newblock {\em Electron. Comm. Probab.}, 7:129--139, 2002.

\bibitem[Arr95]{Ar}
J.~M. Arrieta.
\newblock Rates of eigenvalues on a dumbbell domain. simple eigenvalue case.
\newblock {\em Transactions of the American Mathematical Society}, 347(9):3503--3531, 1995.

\bibitem[BB99]{BB}
R.~Ba{\~n}uelos and K.~Burdzy.
\newblock On the ``hot spots'' conjecture of {J}. {Rauch}.
\newblock {\em J. Funct. Anal.}, 164(1):1--33, 1999.

\bibitem[BD92]{BD}
R.~Ba{\~n}uelos and B.~Davis.
\newblock Sharp estimates for {Dirichlet} eigenfunctions in horn-shaped regions.
\newblock {\em Commun. Math. Phys.}, 150(1):209--215, 1992.

\bibitem[BL76]{BL}
H.~J. Brascamp and E.~H. Lieb.
\newblock On extensions of the {Brunn}-{Minkowski} and {Prekopa}-{Leindler} theorems, including inequalities for log concave functions, and with an application to the diffusion equation.
\newblock {\em J. Funct. Anal.}, 22:366--389, 1976.

\bibitem[CM23]{CM}
P.~Charron and D.~Mangoubi.
\newblock The inner radius of nodal domains in high dimensions.
\newblock {\em arXiv preprint arXiv:2306.00159}, 2023.

\bibitem[CR20]{CR}
A.~M.~H. Chorwadwala and S.~Roy.
\newblock How to place an obstacle having a dihedral symmetry inside a disk so as to optimize the fundamental {Dirichlet} eigenvalue.
\newblock {\em J. Optim. Theory Appl.}, 184(1):162--187, 2020.

\bibitem[Dan03]{Daners}
D.~Daners.
\newblock Dirichlet problems on varying domains.
\newblock {\em Journal of Differential Equations}, 188(2):591--624, 2003.

\bibitem[DNG12]{DNG}
A.L. Delitsyn, B.T. Nguyen, and D.S. Grebenkov.
\newblock Exponential decay of laplacian eigenfunctions in domains with branches of variable cross-sectional profiles.
\newblock {\em The European Physical Journal B}, 85:1--17, 2012.

\bibitem[Gad05]{Gad}
R.~R. Gadyl'shin.
\newblock On the eigenvalues of a ``dumbbell with a thin handle''.
\newblock {\em Izv. Math.}, 69(2):265--329, 2005.

\bibitem[GM18a]{GM2}
B.~Georgiev and M.~Mukherjee.
\newblock Nodal geometry, heat diffusion and {Brownian} motion.
\newblock {\em Anal. PDE}, 11(1):133--148, 2018.

\bibitem[GM18b]{GM}
B.~Georgiev and M.~Mukherjee.
\newblock On maximizing the fundamental frequency of the complement of an obstacle.
\newblock {\em C. R., Math., Acad. Sci. Paris}, 356(4):406--411, 2018.

\bibitem[GM22]{GM3}
B.~Georgiev and M.~Mukherjee.
\newblock Some applications of heat flow to {Laplace} eigenfunctions.
\newblock {\em Commun. Partial Differ. Equations}, 47(4):677--700, 2022.

\bibitem[GN13]{GN}
D.S. Grebenkov and B.T. Nguyen.
\newblock Geometrical structure of laplacian eigenfunctions.
\newblock {\em siam REVIEW}, 55(4):601--667, 2013.

\bibitem[GT83]{GT}
D.~Gilbarg and N.S. Trudinger.
\newblock {\em Elliptic partial differential equations of second order. 2nd ed}, volume 224 of {\em Grundlehren Math. Wiss.}
\newblock Springer, Cham, 1983.

\bibitem[Hay78]{Hay}
W.~K. Hayman.
\newblock Some bounds for principal frequency.
\newblock {\em Appl. Anal.}, 7:247--254, 1978.

\bibitem[Hen06]{He}
A.~Henrot.
\newblock {\em Extremum problems for eigenvalues of elliptic operators}.
\newblock Front. Math. Basel: Birkh{\"a}user, 2006.

\bibitem[Her63]{Hersch}
J.~Hersch.
\newblock The method of interior parallels applied to polygonal or multiply connected membranes.
\newblock {\em Pacific J. Math.}, 13:1229--1238, 1963.

\bibitem[HKK01]{HarellKrugerKurata}
E.~M. Harrell, P.~Kr\"{o}ger, and K.~Kurata.
\newblock On the placement of an obstacle or a well so as to optimize the fundamental eigenvalue.
\newblock {\em SIAM J. Math. Anal.}, 33(1):240--259, 2001.

\bibitem[HSX15]{HuShiXu15}
J.~Hu, Y.~Shi, and B.~Xu.
\newblock The gradient estimate of a {Neumann} eigenfunction on a compact manifold with boundary.
\newblock {\em Chin. Ann. Math., Ser. B}, 36(6):991--1000, 2015.

\bibitem[HV84]{HaVe}
J.~K. Hale and J.~Vegas.
\newblock A nonlinear parabolic equation with varying domain.
\newblock {\em Arch. Ration. Mech. Anal.}, 86:99--123, 1984.

\bibitem[Jer00]{Je}
D.~Jerison.
\newblock Locating the first nodal line in the {Neumann} problem.
\newblock {\em Trans. Am. Math. Soc.}, 352(5):2301--2317, 2000.

\bibitem[Jim93]{Ji}
S.~Jimbo.
\newblock Perturbation formula of eigenvalues in a singularly perturbed domain.
\newblock {\em Journal of the Mathematical Society of Japan}, 45(2):339--356, 1993.

\bibitem[JM20]{JM}
C.~Judge and S.~Mondal.
\newblock Euclidean triangles have no hot spots.
\newblock {\em Ann. Math. (2)}, 191(1):167--211, 2020.

\bibitem[Kes03]{Kesavan}
S.~Kesavan.
\newblock On two functionals connected to the {L}aplacian in a class of doubly connected domains.
\newblock {\em Proc. Roy. Soc. Edinburgh Sect. A}, 133(3):617--624, 2003.

\bibitem[MS21]{MS1}
M.~Mukherjee and S.~Saha.
\newblock On the effects of small perturbation on low energy laplace eigenfunctions.
\newblock {\em arXiv preprint arXiv:2108.13874}, 2021.

\bibitem[MS22]{MS}
M.~Mukherjee and S.~Saha.
\newblock Nodal sets of {Laplace} eigenfunctions under small perturbations.
\newblock {\em Math. Ann.}, 383(1-2):475--491, 2022.

\bibitem[Pay67]{Payne}
L.~E. Payne.
\newblock Isoperimetric inequalities and their applications.
\newblock {\em SIAM Rev.}, 9:453--488, 1967.

\bibitem[Ple56]{Pl}
A.~Pleijel.
\newblock Remarks on {Courant}'s nodal line theorem.
\newblock {\em Commun. Pure Appl. Math.}, 9:543--550, 1956.

\bibitem[P{\'o}l52]{pol}
G.~P{\'o}lya.
\newblock Remarks on the foregoing paper.
\newblock {\em Journal of Mathematics and Physics}, 31(1-4):55--57, 1952.

\bibitem[Ste20]{Ste}
S.~Steinerberger.
\newblock Hot spots in convex domains are in the tips (up to an inradius).
\newblock {\em Commun. Partial Differ. Equations}, 45(6):641--654, 2020.

\bibitem[vdBB99]{vdBB}
M.~van~den Berg and E.~Bolthausen.
\newblock Estimates for {Dirichlet} eigenfunctions.
\newblock {\em J. Lond. Math. Soc., II. Ser.}, 59(2):607--619, 1999.

\bibitem[Veg83]{Ve}
J.~M. Vegas.
\newblock Bifurcations caused by perturbing the domain in an elliptic equation.
\newblock {\em J. Differ. Equations}, 48:189--226, 1983.

\end{thebibliography}


\begin{thebibliography}{A}

\bibitem[Ar]{Ar} J. M. Arrieta,  {\em Rates of eigenvalues on a dumbbell domain. Simple eigenvalue case}, Trans. Amer. Math. Soc., {\bf 347} (1995), no. 9, 3503 – 3531.

\bibitem[AAK]{AAK}  T. V. Anoop,  K. Ashok Kumar, and S. Kesavan, {\em A shape variation result via the geometry of eigenfunctions}, J. Differential Equations, {\bf 298} (2021), 430 – 462.

\bibitem[D]{D} D. Daners, {\em Dirichlet problems on varying domains}, J. Differential Equations, {\bf 188} (2003), no. 2, 591 – 624.

\bibitem[GM]{GM} B. Georgiev and M. Mukherjee, {\em On maximizing the fundamental frequency of the complement of an obstacle}, C. R. Math. Acad. Sci. Paris, {\bf 356} (2018), no. 4, 406 – 411.

\bibitem[GN]{GN} D. Grebenkov and B.-T. Nguyen, {\em Geometrical structure of Laplace eigenfunctions,} SIAM Review, {\bf 55} (2013), no. 4, 601 - 667.

\bibitem[GT]{GT} D. Gilbarg and Neil Trudinger, {\em Elliptic partial differential equations of second order,} Reprint of the 1998 edition. Classics in Mathematics. Springer-Verlag, Berlin, 2001.

\bibitem[HZ]{HZ} A. Henrot and D. Zucco, {\em Optimizing the first Dirichlet eigenvalue of the Laplacian with an obstacle}, Ann. Sc. Norm. Super. Pisa Cl. Sci. (5), {\bf 19} (2019), no. 4, 1535 – 1559.

\bibitem[He]{He} A. Henrot, {\em Extremum problems for eigenvalues of elliptic operators,} Frontiers in Mathematics. Birkhäuser Verlag, Basel, 2006. x+202 pp.

\bibitem[Ji]{Ji} S. Jimbo,{\em Perturbation formula of eigenvalues in a singularly perturbed domain}, J. Math. Soc. Japan, {\bf 45} (1993), no. 2, 339 – 356. 

\bibitem[JK]{JK} S. Jimbo and S. Kosugi, {\em Spectra of domains with partial degeneration}, J. Math. Sci. Univ. Tokyo, {\bf 16} (2009),  269 – 414.

\bibitem[O]{O} S. Ozawa, {\em Singular variation of domains and eigenvalues of the Laplacian}, Duke Math.
J., {\bf 48} (1981), no. 4, 767 – 778.

\bibitem[Siu]{Siu} B. Siudeja, {\em On mixed Dirichlet-Neumann eigenvalues of triangles},  Proc. Amer. Math. Soc., {\bf 144} (2016), no. 6, 2479 - 2493.

\bibitem[HS]{HS} Holcman, D.; Schuss, Z. The narrow escape problem. SIAM Rev. 56 (2014), no. 2, 213–257.

\bibitem[MS1]{MS1} M. Mukherjee and S. Saha, {\em On the effects of small perturbation on low energy Laplace eigenfunctions}, arXiv:2108.13874.

\bibitem[MS2]{MS2} M. Mukherjee and S. Saha, {\em Heat profile, level sets and hot spots of Laplace eigenfunctions}, arXiv:2109.06531. 

\bibitem[RS]{RS} Z. Rudnick and P. Sarnak, The behaviour of eigenstates of arithmetic hyperbolic mani-
folds. Comm. Math. Phys. 161 (1994), no. 1, 195–213.
\bibitem[Li]{Li} E. LINDENSTRAUSS, Invariant measures and arithmetic quantum unique ergodicity, Ann. of
Math. 163 (2006), 165–219. MR 2007b:11072
\bibitem[Sou]{Sou} K. Soundararajan, Quantum unique ergodicity for SL2(Z) H, Ann. of Math. (2) 172
(2010), no. 2, 1529-1538 (arXiv:0901.4060).
\bibitem[FNB]{FNB} F. Faure, S. Nonnenmacher, and S. De Bievre, Scarred eigenstates for quantum cat maps
of minimal periods. Comm. Math. Phys. 239 (2003), no. 3, 449-492.
\bibitem[S2]{S2} C. D. Sogge, Kakeya-Nikodym averages and Lp-norms of eigenfunctions, Tohoku Math.
J. (2) 63 (2011), no. 4, 519-538 (arXiv:0907.4827).

\bibitem[S1]{S1} C. D. Sogge, Fourier integrals in classical analysis. Cambridge Tracts in Mathematics,
105. Cambridge University Press, Cambridge, 1993.

\bibitem[BuGT]{BuGT} N. Burq, P. Gerard, and N. Tzvetkov, Restrictions of the Laplace-Beltrami eigenfunctions
to submanifolds, Duke Math. J. 138 (2007), 445?486. MR 2322684.
\bibitem[SZ1]{SZ1} C.D. Sogge and S. Zelditch, On eigenfunction restriction estimates and L4-bounds for
compact surfaces with nonpositive curvature. Advances in analysis: the legacy of Elias
M. Stein, 447-461, Princeton Math. Ser., 50, Princeton Univ. Press, Princeton, NJ, 2014.
\bibitem[SZ2]{SZ2} C. Sogge and S. Zelditch, Riemannian manifolds with maximal eigenfunction growth,
Duke Math. J. 114 (2002), 387-437.
\bibitem[MO]{MO} @misc{dgdiffer40:online,
author = {},
title = {dg.differential geometry - Normal coordinates near the boundary - MathOverflow},
howpublished = {\url{https://mathoverflow.net/questions/163401/normal-coordinates-near-the-boundary}},
month = {},
year = {},
note = {(Accessed on 06/29/2023)}
}
\bibitem[Fr]{Fr}L. Friedlander, Some inequalities between Dirichlet and Neumann eigenvalues, Arch. Rational Mech.
 Anal. 116 (1991), 153-160

\bibitem[Gad]{Gad} Gadyl'shin, R. R. (2005). On the eigenvalues of a“dumbbell with a thin handle”.Izv.Ross. Akad. Nauk Ser. Mat. 69(2):45–110.


\bibitem[Ve]{Ve} J. M. Vegas, Bifurcation caused by perturbing the domain in an elliptic equation, J.
Differential Equations 48(1983), 189-226.

\bibitem[HaVe]{HaVe} J. K. Hale and J. Vegas, A nonlinear parabolic equation with varying domain, Arch. Rat.
Mech. Anal. 86(1984), 99-123.

\bibitem[HuShiXu15]{HuShiXu15}
\bibitem[AtBu02]{AtBu02}

\bibitem[P]{P} L. E. Payne, {\em Isoperimetric inequalities and their applications}, SIAM Rev., {\textbf{9}} (1967), 453 - 488.

\bibitem[Je]{Je} D. Jerison (2000) Locating the first nodal line in the Neumann problem. Trans. Amer.
Math. Soc. 352, 2301-2317.

\bibitem[Her]{Her} J. Hersch, The method of interior parallels applied to polygonal or multiply connected
membranes, Pacific J. Math. 13 (1963), 1229–1238. Zbl 0127.40603

\bibitem[K]{K} S. Kesavan, On two functionals connected to the Laplacian in a class of doubly connected
domains, Proc. Roy. Soc. Edinburgh Sect. A 133(3) (2003), 617–624. Zbl 1046.35021

\bibitem[HKK]{HKK} E. M. Harrell, P. Kroger, K. Kurata, On the placement of an obstacle or a well so
as to optimize the fundamental eigenvalue, SIAM J. Math. Anal. 33(1) (2001), 240–259.
Zbl 0994.47015
\end{thebibliography}
\bibliographystyle{alpha}

\end{document}